\documentclass[12pt]{amsart}
\usepackage[english]{babel}
\usepackage{amssymb,amscd}

\usepackage[all]{xy}

\usepackage[latin1]{inputenc}

\usepackage{euscript}
\usepackage{amsmath}
\usepackage{ifpdf}

\usepackage{amsfonts}
\usepackage{mathrsfs}

\usepackage{amsmath}

\usepackage{mathpazo}

\newtheorem{theorem}{Theorem}[section]
\newtheorem{lemma}[theorem]{Lemma}
\newtheorem{corollary}[theorem]{Corollary}
\newtheorem{proposition}[theorem]{Proposition}
\newtheorem{remarks}[theorem]{Remarks}
\newtheorem{remark}[theorem]{Remark}
\newtheorem{example}[theorem]{Example}
\newtheorem{examples}[theorem]{Examples}
\newtheorem{definition}[theorem]{Definition}

\newtheorem{question}[theorem]{Question}
\newtheorem{convention}[theorem]{Convention}

\def\sqr#1#2{{\vcenter{\hrule height.#2pt
\hbox{\vrule width.#2pt height#1pt \kern#1pt \vrule width.#2pt}
\hrule height.#2pt}}}

\def\qed{\hspace*{\fill} $\square$}

\def\Mc{\mathscr{M}}
\def\Nc{\mathscr{N}}

\def\Cc{\mathscr{C}}

\def\Ic{\mathscr{I}}

\def\Kc{\mathscr{K}}

\begin{document}

\title[Generalized Ulrich modules]{On the theory of generalized Ulrich modules}

\author[Miranda-Neto, \, Queiroz, \, Souza]{Cleto B.~Miranda-Neto, \, Douglas S.~Queiroz, \, Thyago S.~Souza}

\address{Departamento de Matem\'atica, Universidade Federal da
Para\'iba, 58051-900, Jo\~ao Pessoa, PB, Brazil.}
\email{cleto@mat.ufpb.br}

\address{Departamento de Matem\'atica, Universidade Federal da
Para\'iba, 58051-900, Jo\~ao Pessoa, PB, Brazil.}
\email{douglassqueiroz0@gmail.com}

\address{Unidade Acad\^emica de Matem\'atica, Universidade Federal de
Campina Grande, 58429-970, Campina Grande, PB, Brazil.}
\email{thyago@mat.ufcg.edu.br}


\begin{abstract}
In this paper we further develop the theory of generalized Ulrich modules introduced in 2014 by Goto {\it et al.} Our main goal is to address the problem of when the operations of taking the Hom functor and horizontal linkage preserve the Ulrich property. One of the applications is a new characterization of quadratic
hypersurface rings. Moreover, in the Gorenstein case, we deduce that applying linkage to sufficiently high syzygy modules of Ulrich ideals yields Ulrich modules. Finally, we explore connections to the theory of modules with minimal multiplicity, and as a byproduct we determine the Chern number of an Ulrich module as well as the Castelnuovo-Mumford regularity of its Rees module.

\end{abstract}

\maketitle

\vspace{-0.1in}

\centerline{\it Dedicated with gratitude to the memory of Professor Shiro Goto.}

\medskip

\section{Introduction}

In this work we are concerned with the theory of generalized Ulrich modules (over Cohen-Macaulay local rings) by Goto et al.\,\cite{goto2014}, which widely extended the classical study of maximally generated maximal Cohen-Macaulay modules -- or {\it Ulrich modules}, as coined in \cite{H-K} -- initiated in the 80's by B. Ulrich \cite{ulrich84}. The term {\it generalized} refers to the fact that Ulrich modules are taken relatively to a zero-dimensional ideal which is not necessarily the maximal ideal, the latter situation corresponding to the classical theory; despite the apparent naivety of the idea, this passage adds considerable depth to the theory and enlarges its horizon of applications. 

Motivated by the remarkable advances in \cite{goto2014}, our purpose here is to present further progress which includes generalizations of several known results on Ulrich modules (e.g., from \cite{goto2014}, \cite{kobayashi2018}, \cite{Ooishi}, \cite{ulrich2003}) and connections to some other important classes such as that of modules with minimal multiplicity; for the latter task, we employ suitable numerical invariants such as the Castelnuovo-Mumford regularity of blowup modules.

It is worth recalling that the original notion of an Ulrich module (together with the classical existence problem; see, however, \cite{Yhee} for Yhee's construction of local domains which do not admit Ulrich modules or (weakly) lim Ulrich sequences) has been extensively explored since its inception, in both commutative algebra and algebraic geometry; see, for instance,
\cite{BHU}, \cite{hartshorne2011}, \cite{ESW}, \cite{G-T-T0}, \cite{Ha-Hu}, \cite{H-K}, \cite{ulrich91}, \cite{KleMir}, \cite{Ma}, \cite{yoshida2017}, and their references on the theme. Echoing and complementing \cite[Second paragraph of Introduction]{Yhee}, the applications include criteria for the Gorenstein property (see \cite{Ha-Hu}, \cite{ulrich84}), the investigation of maximal Cohen-Macaulay modules over Gorenstein local rings and factoriality of certain rings (see \cite{H-K}),
the development of the theory of almost Gorenstein rings \cite{G-T-T0}, strategies to tackle certain resistant conjectures in multiplicity theory (e.g., Lech's conjecture; see \cite{Ma}, where Ma solved this conjecture in the graded case by introducing and using the notion of (weakly)
lim Ulrich sequences, which gives yet another way to generalize the classical Ulrich property), and methods for constructing resultants and Chow forms of projective algebraic varieties (see \cite{ESW}, where the concepts of Ulrich sheaf and Ulrich bundle were introduced).

In essence, the general approach suggested in \cite{goto2014} extended the definition of an Ulrich
module $M$ over a (commutative, Noetherian) Cohen-Macaulay local ring $(R, \Mc)$ with infinite residue field to a relative setting that takes into account an $\Mc$-primary ideal $\Ic$ containing a parameter ideal as a reduction, so that the case $\Ic=\Mc$ retrieves the standard theory. For instance, it is now required the condition of the freeness of $M/\Ic M$ over $R/\Ic$, which was hidden in the classical setting as $M/\Mc M$ is simply a vector space. Following this line of investigation, other works have appeared in the literature as for example \cite{G-I-K}, \cite{goto2016}, \cite{G-T-T}, \cite{Num}.

Now let us briefly comment on our main results, section by section. Preliminary definitions and some known auxiliary results, which are used throughout the paper, are given in
Section \ref{aux}.

The main goal of Section \ref{sec3} is to investigate the Ulrich property under the Hom functor. In this regard, our main
result is Theorem \ref{teo.teo5.1.goto.generalizado}, which can be viewed as a generalization of \cite[Theorem 5.1]{goto2014} and of
\cite[Proposition 4.1]{kobayashi2018}. Moreover, Corollary \ref{cor.teo5.2.goto.generalizado} generalizes \cite[Corollary 5.2]{goto2014}, and Corollary \ref{far} is a far-reaching extension of \cite[Lemma 2.2]{BHU}. We also study a connection to the theory of semidualizing modules (see Corollary \ref{cor3.teo5.1.goto.generalizado}) and use it to derive a new characterization of when $R$ is regular (see Corollary \ref{regularity}). In addition, in the last subsection, we provide some freeness criteria for $M/\Ic M$ over the Artinian local ring $R/\Ic$, which is one of the requirements for Ulrichness with respect to $\Ic$.

In Section \ref{sec4} we are essentially interested in the behavior of the Ulrich property under the operation of horizontal linkage over Gorenstein local rings. The main result here is Theorem \ref{teo.link.ulrich} (see also Corollary \ref{prop.I.Ulrich.preservada.por.LH}), from which we derive a curious characterization of quadratic hypersurface local rings (see Corollary \ref{Hyper-charact}). In Corollary \ref{lambda-syz}, we record the special case of sufficiently high syzygy modules of a non-parameter Ulrich ideal, in case $R$ is Gorenstein.

In Section \ref{sec5} we consider the class of modules with minimal multiplicity (in the sense of \cite{Puthenpurakal2}) and then connect this concept to the Ulrich property, both taken with respect to $\Ic$. The basic relation is that Ulrich $R$-modules have minimal multiplicity (see Proposition \ref{UlrichMultMin}), and as a consequence we use the Chern number -- i.e., the first Hilbert coefficient -- as an ingredient to obtain a characterization of Ulrichness (see Corollary \ref{UlrHilb}) which generalizes \cite[Corollary 1.3(1)]{Ooishi}. Under this perspective, modules with trivial Chern number are provided in  Corollary \ref{Chern-of-syz}, and considerations about the structure of the Hilbert-Samuel polynomial of an Ulrich module are given in Remarks \ref{Hilb-S}. Our main technical result in this section is Theorem \ref{regMulMin}, which curiously does not contain Ulrich-like properties in its statement and, more precisely, characterizes modules with minimal multiplicity as follows:

\medskip

\noindent{\bf Theorem 5.14} {\it Let $(R, \Mc)$ be a Noetherian local ring with infinite residue field, $M$ a Cohen-Macaulay $R$-module of dimension $t > 0$ and $I$ an $\Mc$-primary ideal of $R$. Let $J = (z_{1}, \ldots,z_{t})$ be a minimal $M$-reduction of $I$. The following assertions are equivalent:
\begin{itemize}
\item[(i)] 	$M$ has minimal multiplicity with respect to $I$;
\item[(ii)] $\mathrm{reg}\,{\mathcal R}(I, M) = \mathrm{reg}\,{\mathcal G}(I,M) = {\rm r}_{J}(I, M)  \leq 1$;
\item[(iii)] ${\rm r}_{J}(I, M)  \leq 1$.
\end{itemize}}

\medskip

Here, ${\rm reg}(-)$ denotes (Castelnuovo-Mumford) regularity, and ${\mathcal R}(I, M)$ (resp. ${\mathcal G}(I,M)$) stands for the Rees module (resp. the associated graded module) of $I$ relative to $M$. Also, ${\rm r}_{J}(I, M)$ is the reduction number of $I$ with respect to $J$ relative to $M$. We emphasize that Theorem \ref{regMulMin} answers affirmatively the module-theoretic analogue of Sally's question (about independence of reduction numbers; see \cite{Sally}) for the class of modules with minimal multiplicity. Additionally, from this theorem we derive Corollary \ref{corUlr}, which determines the regularity of the Rees and the associated graded modules of $\Ic$ relative to an Ulrich module (this result partially generalizes \cite[Proposition 1.1]{Ooishi}), and also Corollary \ref{corUlr2}, where we deal once again with high syzygy modules of Ulrich ideals.

Finally, Section \ref{detailed} provides a detailed example to illustrate some of our main corollaries.

\section{Conventions, preliminaries, and some auxiliary results}\label{aux}

Throughout this paper, all rings are assumed to be commutative and Noetherian with 1, and by {\it finite} module we mean a finitely generated module. 

In this section, we recall some of the basic notions and tools that will play an important role throughout the paper. Other auxiliary notions will be introduced as they become necessary.

\subsection{Ulrich ideals and modules}\label{ulrich-subsect}
Let $(R, \Mc)$ be a local ring, $M$ a finite $R$-module, and $I\neq R$ an ideal of definition of $M$, i.e., ${\Mc}^nM\subset IM$ for some $n>0$. Let us establish some notations. We denote by $\nu(M)$ and
$\textrm{e}_I^0 (M)$, respectively, the minimal number of generators of $M$ and the multiplicity of $M$ with respect to $I$. When $I
= \Mc$, we simply write $\textrm{e}(M)$ in place of $\textrm{e}_{\Mc}^0(M)$.

\begin{definition}\label{def.ulrich.module} \rm Let $(R, \Mc)$ be a local ring. A finite $R$-module $M$ is {\it Cohen-Macaulay} (resp. {\it maximal
Cohen-Macaulay}) if ${\rm depth}_RM={\rm dim}\,M$ (resp. ${\rm depth}_RM={\rm dim}\,R$). Note the zero module is not maximal Cohen-Macaulay (as its depth is set to be $+\infty$). Moreover, $M$ is called \emph{Ulrich} if $M$ is a maximal Cohen-Macaulay $R$-module satisfying $\nu(M)=\textrm{e}(M)$.
\end{definition}

For instance, if $(R, \Mc)$ is a 1-dimensional Cohen-Macaulay local ring, then the power ${\Mc}^{{\rm e}(R)-1}$ is an Ulrich module. Several other classes of examples can be found, e.g., in \cite{BHU}.

Ulrich modules are also dubbed \emph{maximally generated
maximal Cohen-Macaulay modules}. This is due to the fact that there is an inequality
$\nu(M) \leq \textrm{e}(M)$ whenever the local ring $R$ is Cohen-Macaulay and $M$ is maximal
Cohen-Macaulay (see \cite[Proposition 1.1]{BHU}).

\begin{convention}\label{convention}\rm Henceforth, in the entire paper, we adopt the following convention and notations. Whenever $(R, \Mc)$ is a $d$-dimensional Cohen-Macaulay local ring, we will let  $\mathscr{I}$ (to be distinguished from the notation $I$) stand for an $\Mc$-primary ideal that contains a parameter ideal $${Q} = ({\bf x})=(x_1, \ldots, x_d)$$ as a reduction, i.e., $Q\Ic^r=\Ic^{r+1}$ for some integer $r\geq 0$. As is well-known, any $\Mc$-primary ideal of $R$ has this property provided that the residue class field $R/\Mc$ is infinite, or that $R$ is analytically irreducible with $d = 1$.
\end{convention}

\begin{definition}\label{def.Gor.ideal}\rm
Let $R$ be a Cohen-Macaulay local ring. We say that the ideal $\mathscr{I}$ is \emph{Gorenstein} if the quotient ring $R/\mathscr{I}$ is Gorenstein.
\end{definition}

Next, we recall the general notions of Ulrich ideal and Ulrich module as introduced in \cite{goto2014} (where in addition several explicit examples are given). As will be made clear, the latter (Definition \ref{def.I.urich.module} below) generalizes Definition \ref{def.ulrich.module}.

\begin{definition}\label{def.urich.ideal}\rm (\cite{goto2014})
Let $R$ be a Cohen-Macaulay local ring. 
We say that the ideal $\mathscr{I}$
is \emph{Ulrich} if $\mathscr{I}^2 =
{Q}\mathscr{I}$ (i.e., the reduction number of $\mathscr{I}$ with respect to $Q$ is at most 1) and $\mathscr{I}/\mathscr{I}^2$ is a free $R/\mathscr{I}$-module.
\end{definition}

\begin{examples}\rm \label{exemplo} (i) (\cite[Proposition 3.9]{Ku}) Let $S=K[[x, y, z]]$ be a formal power series ring over an infinite field $K$, and fix any regular sequence $\{f, g, h\}\subset (x, y, z)$. Then, $R=S/(f^2-gh, g^2-hf, h^2-fg)$ is a 1-dimensional Cohen-Macaulay local ring and $\Ic=(f, g, h)R$ is an Ulrich ideal.

\medskip

(ii) (\cite[Example 2.7(2)]{goto2014}) One way to produce examples in arbitrary positive dimension is as follows. Given a field $K$ and integers $d, s\geq 1$, consider the $d$-dimensional local hypersurface ring $R=K[[z_1, \ldots, z_{d+1}]]/(z_1^2+\ldots + z_d^2+z_{d+1}^{2s})$, where $z_1, \ldots, z_{d+1}$ are formal indeterminates over $k$. Then, the ideal $$\Ic=(z_1, \ldots, z_d, z_{d+1}^{s})R$$ is Ulrich and contains the parameter ideal $Q=(z_1, \ldots, z_d)R$ as a reduction.

\end{examples}

\begin{remark}\label{obs1.def.urich.ideal}\rm
In a Gorenstein local ring, every Ulrich ideal is Gorenstein (see \cite[Corollary 2.6]{goto2014}).
\end{remark}

\begin{definition}\label{def.I.urich.module}\rm (\cite{goto2014})
Let $R$ be a Cohen-Macaulay local ring and let $M$ be a finite
$R$-module. We say that $M$ is \emph{Ulrich
with respect to} $\mathscr{I}$ if the following conditions hold:
\begin{itemize}
\item[(i)] $M$ is a maximal Cohen-Macaulay $R$-module;
\item[(ii)] $\mathscr{I}M = {Q}M$;
\item[(iii)] $M/\mathscr{I}M$ is a free $R/\mathscr{I}$-module.
\end{itemize}
\end{definition}

\begin{remarks}\label{obs.I.urich.module}\rm (i) Let us recall the discussion in \cite[Paragraph after Definition 1.2]{goto2014}. Denote length of $R$-modules by $\ell_R(-)$.
If $R$ is a Cohen-Macaulay local ring with infinite residue field and $M$ is a
maximal Cohen-Macaulay $R$-module, then
$$\textrm{e}_{\mathscr{I}}^0 (M) = \textrm{e}_{{Q}}^0 (M) = \ell_R(M/{Q}M) \geq \ell_R(M/\mathscr{I}M),$$
so that condition (ii) of Definition \ref{def.I.urich.module} is equivalent to saying that the equality $\textrm{e}_{\mathscr{I}}^0 (M) =
\ell_R(M/\mathscr{I}M)$ takes place. In particular, if $\mathscr{I} = \Mc$, condition (ii) is the same as  $\textrm{e}(M) = \nu(M)$. Therefore, $M$ is an
Ulrich module with respect to $\Mc$ if and only if $M$
is an Ulrich module in the sense of Definition
\ref{def.ulrich.module}.

\medskip

(ii) Clearly, if $d=1$ and $\Ic$ is an Ulrich ideal of $R$, then $\Ic$ is an Ulrich $R$-module with respect to $\Ic$.

\medskip

(iii) Let us recall the following more general recipe to obtain Ulrich modules from Ulrich ideals (in the setting of Convention \ref{convention}). If $\Ic$ is an Ulrich ideal of $R$ which is not a parameter ideal, then for any $i\geq d$ the $i$-th syzygy module (see Subsection \ref{linkage-subsect} below) of $R/\Ic$ is an Ulrich $R$-module with respect to $\Ic$, and conversely (we refer to \cite[Theorem 4.1]{goto2014}). This is a very helpful property and will be explored in some of our results and examples.

\end{remarks}

\subsection{Linkage}\label{linkage-subsect}

The concepts recalled in this subsection can be described in the general context of semiperfect rings, but in this paper we focus on the special case of (finite modules over) a local ring $R$, since this is the setup where our results will be proved.

Given a finite $R$-module 
$M$, we write $M^*={\rm Hom}_R(M, R)$.
The ({\it Auslander}) \emph{transpose} $\textrm{Tr}\,M$ of 
$M$ is defined as the cokernel of the dual $\partial_1^*
= \textrm{Hom}_{R} (\partial_1, R)$ of the first differential map
$\partial_1$ in a minimal free resolution of $M$ over $R$. Hence there is
an exact sequence $$0 \longrightarrow M^{*} \longrightarrow
F_0^{*} \stackrel{\partial_1^{*}}{\longrightarrow} F_1^{*} \longrightarrow
{\rm Tr}\,M \longrightarrow 0$$ for suitable finite free $R$-modules $F_0, F_1$. The ({\it first}) \emph{syzygy} module $\Omega^1M=\Omega M$
of $M$ is the image of $\partial_1$, hence a submodule of $F_0$. We recursively put $\Omega^{k}M=\Omega (\Omega^{k-1}M)$, the $k$-th syzygy module of $M$, for any $k\geq 2$.

Note that the modules $\textrm{Tr}\,M$ and
$\Omega M$ are uniquely determined up to isomorphism, since so is
a minimal free resolution of $M$. By \cite[Proposition
6.3]{auslander65}, we have an exact sequence
\begin{equation}\label{eq.seq.exact}
0 \longrightarrow \textrm{Ext}_{R}^1 (\textrm{Tr}\,M, R) \longrightarrow M
\stackrel{e_M}{\longrightarrow} M^{**}
             \longrightarrow \textrm{Ext}_{R}^2(\textrm{Tr}\,M, R) \longrightarrow 0,
\end{equation}
where $e_M$ is the evaluation map.

In \cite{link2004}, Martsinkovsky and Strooker generalized the classical theory of linkage for ideals to the
context of modules by means of the operator $\lambda  = \Omega \textrm{Tr}$, i.e., a finite $R$-module $M$ is sent to the composite $\Omega \textrm{Tr}\,M$ defined from a minimal free presentation of $M$.

\begin{definition}\rm (\cite{link2004})  Two finite $R$-modules $M$ and $N$ are said to be \emph{horizontally
linked} if $M \cong \lambda N$ and $N \cong \lambda M$. In the case where $M$ and $\lambda M$ are horizontally linked, i.e., $M \cong \lambda^2 M$,
we simply say that the module $M$ is horizontally linked.

\end{definition}

Also we recall that a \emph{stable} module is a finite module with no non-zero free direct
summand. A finite $R$-module $M$ is called a \emph{syzygy module} if it
is embedded in a finite free $R$-module, i.e., if $M\cong \Omega N$ for some finite $R$-module $N$. Here is a well-known characterization of
horizontally linked modules.

\begin{lemma}\label{teo2.link} {\rm (\cite[Theorem 2 and Corollary 6]{link2004})}
A finite $R$-module $M$ is horizontally linked if and only if $M$ is
stable and $\textrm{\emph{Ext}}_{R}^1(\textrm{\emph{Tr}}\,M, R) =
0$, if and only if $M$ is a stable  syzygy module.
\end{lemma}

\begin{lemma}\label{prop4.link} {\rm (\cite[Proposition 4]{link2004})}
Suppose $M$ is horizontally linked. Then, $\lambda M$ is also
horizontally linked and, in particular, $\lambda M$ is stable.
\end{lemma}

\subsection{Canonical modules}

In the following we collect basic facts about canonical modules.

\begin{lemma}\label{teo.canonical.module} {\rm (\cite{CMr})} Let $R$ be a Cohen-Macaulay local ring with canonical module $\omega_R$. Let $M$ be a maximal Cohen-Macaulay $R$-module. Then the
following statements hold:
\begin{itemize}
\item[(i)] $\textrm{\emph{Hom}}_{R}(M, \omega_R)$ is a maximal Cohen-Macaulay $R$-module; 
\item [(ii)] $\textrm{\emph{Ext}}_{R}^i(M, \omega_R)=0$ for all $i > 0$;
\item [(iii)] $M \cong \textrm{\emph{Hom}}_{R}(\textrm{\emph{Hom}}_{R} (M, \omega_R), \omega_R)$;
\item [(iv)] If ${\bf y}$ is an $R$-sequence, then $R/({\bf y})$ has a canonical module
$\omega_{R/({\bf y})} \cong\omega_R/{\bf y}\omega_R$; 
\item [(v)] Let $\varphi: R\rightarrow S$
be a local homomorphism of Cohen-Macaulay local rings such that $S$ is a finite $R$-module. Then $S$ has a canonical module $\omega_S \cong \textrm{\emph{Ext}}_{R}^{t}(S,
\omega_R)$, where $t = \dim R - \dim S$. 
\end{itemize}
\end{lemma}

\section{Hom functor and the Ulrich property} \label{sec3}

In this section we investigate, in essence, the behavior of the Ulrich property under the Hom functor.

\subsection{Key lemma, main result, and corollaries} We start with the following basic lemma, which will be a key ingredient in the proof of the main result of this section.

\begin{lemma} \label{lema2.geral}
Let $R$ be a Cohen-Macaulay local ring, $M, N$ be maximal
Cohen-Macaulay $R$-modules, and ${\bf y} = y_1, \ldots, y_n$ be an
$R$-sequence for some $n\geq 1$.
\begin{itemize}
\item[(i)] If either $n=1$ or $\textrm{\emph{Ext}}_{R}^i(M, N)=0$ for all $i=1, \ldots, n-1$, there is an injection $$\textrm{\emph{Hom}}_{R}(M, N)/{\bf y}\textrm{\emph{Hom}}_{R}(M, N) \hookrightarrow
\textrm{\emph{Hom}}_{R/({\bf y})}(M/{\bf y}M, N/{\bf y}N);$$
\item[(ii)] If $\textrm{\emph{Ext}}_{R}^i(M, N)=0$ for all $i=1, \ldots, n$, there is an isomorphism $$\textrm{\emph{Hom}}_{R} (M, N)/{\bf y}\textrm{\emph{Hom}}_{R}(M, N) \cong
\textrm{\emph{Hom}}_{R/({\bf y})}(M/{\bf y}M,
N/{\bf y}N).$$
\end{itemize}
\end{lemma}
\begin{proof} We shall prove the assertion (i), which from the arguments below (essentially from (\ref{eq1.lema.})) will be easily seen to imply (ii). Set $R' = R/(y_1)$, $M' = M / y_1 M$, and
$N' = N / y_1 N$. We will proceed by induction on $n$. Consider first the case $n = 1$, which is standard but we supply the proof for convenience. Since $M$ and $N$ are maximal Cohen-Macaulay $R$-modules and $y_1 \in \Mc$ is 
$R$-regular (where $\Mc$ is the maximal ideal of $R$), it follows that $y_1$ is both $M$-regular and $N$-regular. In particular, we have the short exact sequence
$$0 \longrightarrow M \stackrel{y_1}{\longrightarrow} M \longrightarrow M'
\longrightarrow 0,$$
which induces the exact sequence
\begin{equation}\label{eq1.lema.}
\begin{array}{c} 0 \longrightarrow \textrm{Hom}_R(M', N)
\longrightarrow \textrm{Hom}_R(M, N) \stackrel{y_1}{\longrightarrow}
\textrm{Hom}_{R} (M, N) \longrightarrow \textrm{Ext}_{R}^{1}(M', N)
\longrightarrow  \\
 \cdots \longrightarrow
\textrm{Ext}_R^i(M, N) \longrightarrow \textrm{Ext}_R^{i+1}(M', N)
\longrightarrow \textrm{Ext}_R^{i+1}(M, N) \longrightarrow \cdots.
\end{array}
\end{equation}
It follows an injection
\begin{equation}\label{eq2.lema.}
\textrm{Hom}_{R} (M, N) /y_1 \textrm{Hom}_{R} (M, N)
\hookrightarrow \textrm{Ext}_{R}^{1}(M',N).
\end{equation}
Because $y_1$ is $N$-regular and $y_1 M' = 0$, there are isomorphisms (see \cite[Lemma 3.1.16]{CMr})
\begin{equation}\label{eq3.lema.}
\textrm{Ext}_{R'}^i(M', N') \cong \textrm{Ext}_{R}^{i+1}(M', N) \quad
\mbox{for \,all} \quad i \geq 0.
\end{equation}
In particular,
\begin{equation}\label{eq4.lema.}
\textrm{Hom}_{R'}(M', N') \cong \textrm{Ext}_{R}^{1}(M', N),
\end{equation}
and the result follows by (\ref{eq2.lema.}) and
(\ref{eq4.lema.}).

Now let $n \geq 2$. Clearly, $R'$ is a Cohen-Macaulay ring and $M', N'$
are maximal Cohen-Macaulay $R'$-modules. By assumption,
$\textrm{Ext}_{R}^i(M, N)=0$ for all $i=1, \ldots, n-1$. Thus, using
(\ref{eq1.lema.}) and  (\ref{eq3.lema.}), we
obtain isomorphisms
\begin{equation}\label{eq5.lema.}
\textrm{Ext}_{R'}^{i}(M', N') \cong \left\{
\begin{array}{ccccccc}
\textrm{Hom}_{R} (M, N) /y_1 \textrm{Hom}_{R} (M, N),  \textrm{ if } i = 0, \\
                0,  \textrm{ if } i=1, \ldots, n-2.
\end{array}\right.
\end{equation}
Since $\textbf{y}' = y_2, \ldots, y_n$ is an $R'$-sequence, the
induction hypothesis yields an injection $$\textrm{Hom}_{R'} (M', N') /
\textbf{y}' \textrm{Hom}_{R'} (M', N') \hookrightarrow
\textrm{Hom}_{R'/\textbf{y}'R'}(M'/\textbf{y}' M', N'/ \textbf{y}'
N'),$$ where the latter module is clearly isomorphic to ${\rm Hom}_{R/(\textbf{y})}(M/\textbf{y}M, N/\textbf{y}N)$. Now the conclusion follows by (\ref{eq5.lema.}) with
$i=0$. \qed
\end{proof}

\medskip

The theorem below is our main result in this section.

\begin{theorem}\label{teo.teo5.1.goto.generalizado}
Let $R$ be a Cohen-Macaulay local ring of dimension $d$. Let $M,N$
be maximal Cohen-Macaulay $R$-modules such that
$\textrm{\emph{Hom}}_{R} (M, N) \neq 0$ and
$\textrm{\emph{Ext}}_{R}^i(M, N)=0$ for all $i=1, \ldots, n$, where either $n=d-1$ or $n=d$. Let $\Ic$ and $Q$ be as in Convention \ref{convention}. Assume that $M$ {\rm (}resp.\,$N${\rm )} is an Ulrich $R$-module with respect
to $\Ic$, and consider the following conditions:
\begin{itemize}
\item[(i)] $\textrm{\emph{Hom}}_{R} (M, N)$ is an Ulrich
$R$-module with respect to $\Ic$;
\item[(ii)] $\textrm{\emph{Hom}}_{R} (M, N) / \Ic \textrm{\emph{Hom}}_{R} (M, N)$ is a free $R/\Ic$-module;
\item[(iii)] $\textrm{\emph{Hom}}_{R/Q}(R/\Ic, N/QN)$ {\rm (}resp. $\textrm{\emph{Hom}}_{R/Q}(M/QM, R/\Ic)${\rm )} is a free $R/\Ic$-module.
\end{itemize}
Then the following statements hold:
\begin{enumerate}
\item[(a)] If $n = d-1$ then {\rm (i)} $\Leftrightarrow$ {\rm (ii)};
\item[(b)] If $n = d$ then {\rm (i)} $\Leftrightarrow$ {\rm (ii)} $\Leftrightarrow$ {\rm (iii)}.
\end{enumerate}
\end{theorem}
\begin{proof}
(a) Applying the functor ${\rm Hom}_R(-, N)$ to a free resolution
$$\cdots \longrightarrow F_{d+1} \longrightarrow F_{d} \longrightarrow F_{d-1}
\longrightarrow \cdots \longrightarrow F_{1} \longrightarrow F_0 \longrightarrow M
\longrightarrow 0$$ of the $R$-module $M$, and using the hypothesis that $\textrm{Ext}_{R}^i(M, N)=0$ for $i=1, \ldots, d-1$, we obtain an exact sequence $$0 \rightarrow
\textrm{Hom}_{R}(M, N) \rightarrow \textrm{Hom}_{R}(F_0, N)
\rightarrow \cdots \rightarrow \textrm{Hom}_{R}(F_{d-1}, N)
\rightarrow \textrm{Hom}_{R}(F_{d}, N).$$ Now set $X_0 :=
\textrm{Hom}_{R} (M, N)$ and $X_i := \textrm{Im} 
(\textrm{Hom}_{R} (F_{i-1}, N) \rightarrow \textrm{Hom}_{R}
(F_{i}, N))$, $i = 1, \ldots, d$. Since $N$ is maximal
Cohen-Macaulay, we have $\textrm{depth}_R\textrm{Hom}_{R}
(F_{i}, N) = d$ for all $i = 0, \ldots, d$. Thus, by the short exact sequence
$$0 \longrightarrow X_i \longrightarrow
\textrm{Hom}_{R}(F_{i}, N) \longrightarrow  X_{i+1} \longrightarrow
0,$$ we get $\textrm{depth}_RX_i \geq \min \{ d, \textrm{depth}_RX_{i+1} + 1 \}$ (see, e.g., \cite[Proposition 1.2.9]{CMr}). Therefore, $$\textrm{depth}_R\textrm{Hom}_{R} (M,N) \geq
\min \{ d, \textrm{depth}_RX_{d} + d \} = d,$$ i.e., $\textrm{Hom}_{R} (M,N)$ is a maximal Cohen-Macaulay $R$-module.

Now, as in Convention \ref{convention}, let ${\bf x}=x_1, \ldots, x_d$ be a generating set of the parameter ideal $Q$. Then ${\bf x}$ is an $R$-sequence (see \cite[Theorem 2.1.2(d)]{CMr}), and so by Lemma
\ref{lema2.geral}(i) there is an injection
\begin{equation}\label{eq1.teo.teo5.1.goto.generalizado}
\textrm{Hom}_{R} (M, N)/Q\textrm{Hom}_{R} (M, N)
\hookrightarrow \textrm{Hom}_{R/Q} (M/QM,
N/QN).
\end{equation}
Because $M$ (resp. $N$) is assumed to be Ulrich with respect to $\Ic$, the module $M/QM$ (resp. $N/QN$) is annihilated by
$\Ic$, and hence so is $\textrm{Hom}_{R/Q}(M/QM, N/QN)$. In either case, it follows from
(\ref{eq1.teo.teo5.1.goto.generalizado}) that the quotient 
$\textrm{Hom}_{R} (M, N)/Q\textrm{Hom}_{R} (M, N)$ is
annihilated by $\Ic$. Thus, $$\Ic\textrm{Hom}_{R}(M, N) = Q\textrm{Hom}_{R}(M, N).$$ Therefore,
$\textrm{Hom}_{R} (M, N)$ is  Ulrich with respect to
$\Ic$ if and only if the quotient module
$\textrm{Hom}_{R} (M, N)/\Ic \textrm{Hom}_{R} (M, N)$ is
$R/\Ic$-free, i.e., (i) $\Leftrightarrow$ (ii).

(b) As seen above, there is an equality $\Ic\textrm{Hom}_{R} (M, N) =
Q\textrm{Hom}_{R} (M, N)$. Notice that, even more, Lemma \ref{lema2.geral}(ii) yields an isomorphism
\begin{equation}\label{eq2.teo.teo5.1.goto.generalizado}
\textrm{Hom}_{R} (M, N)/Q\textrm{Hom}_{R}(M, N) \cong
\textrm{Hom}_{R/Q}(M/QM, N/QN).
\end{equation} 
Now suppose, say, $M$ is Ulrich with respect to $\Ic$.
From $M/QM=M/\Ic M \cong
(R/\Ic)^m$ for some integer $m
> 0$, we deduce that
\begin{equation}\label{eq3.teo.teo5.1.goto.generalizado}
\textrm{Hom}_{R/Q} (M/QM, N/QN)
\cong (\textrm{Hom}_{R/Q} (R/\Ic,
N/QN))^m.
\end{equation}
By (\ref{eq2.teo.teo5.1.goto.generalizado}) and
(\ref{eq3.teo.teo5.1.goto.generalizado}), we get $$\textrm{Hom}_{R} (M, N) / \Ic \textrm{Hom}_{R} (M, N)
\cong (\textrm{Hom}_{R/Q} (R/\Ic,
N/QN))^m. $$ Therefore, the quotient $\textrm{Hom}_{R} (M, N) /\Ic \textrm{Hom}_{R} (M, N)$ is $R/\Ic$-free if and only if the module $\textrm{Hom}_{R/Q}(R/\Ic, N/QN)$ is $R/\Ic$-free. The case where $N$ is Ulrich with respect
to $\Ic$ is completely similar. This shows (ii) $\Leftrightarrow$ (iii) and concludes the proof of the theorem. \qed
\end{proof}

\begin{remark}\rm It is worth observing that the condition $\textrm{{Hom}}_{R}(M, N) = 0$ can hold even if $M$ and $N$ are both maximal Cohen-Macaulay. For instance, over the local ring $R=K[[x, y]]/(xy)$, where $x, y$ are formal variables over a field $K$, we have $$\textrm{{Hom}}_{R}(R/xR, R/yR) = 0.$$ We do not know whether this can occur if $M$ or $N$ is Ulrich.
    
\end{remark}

We point out that Theorem \ref{teo.teo5.1.goto.generalizado} generalizes \cite[Proposition
4.1]{kobayashi2018} (see Corollary \ref{cor2.teo5.1.goto.generalizado}, to be given shortly) and, in addition, recovers the following result from \cite{goto2014}.

\begin{corollary}\label{cor1.teo5.1.goto.generalizado}
{\rm (\cite[Theorem 5.1]{goto2014})} Let $R$ be a Cohen-Macaulay local
ring with canonical module $\omega_R$, and let $M$ be an Ulrich
$R$-module with respect to $\Ic$. Then the following
assertions are equivalent:
\begin{itemize}
\item[(i)] $\textrm{\emph{Hom}}_{R}(M, \omega_R)$ is an Ulrich $R$-module with respect to $\Ic$;
\item[(ii)] $\Ic$ is a Gorenstein ideal.
\end{itemize}
\end{corollary}
\begin{proof}
By Lemma \ref{teo.canonical.module}(ii), we have $\textrm{Ext}_{R}^i(M,
\omega_R)=0$ for all $i > 0$. Since $R/Q$ and
$R/\Ic$ are zero-dimensional local rings and the ideal $Q$ is generated by an $R$-sequence, there are isomorphisms
$$\omega_{R/\Ic} \cong \textrm{Hom}_{R/Q}
(R/\Ic, \omega_{R/Q}) \cong
\textrm{Hom}_{R/Q} (R/\Ic, \omega_R/Q\omega_R)$$ according to standard facts (see parts (iv) and (v) of Lemma \ref{teo.canonical.module}). Now, applying Theorem \ref{teo.teo5.1.goto.generalizado}(b) with $N = \omega_R$, we derive that
$\textrm{Hom}_{R}(M, \omega_R)$ is Ulrich with
respect to $\Ic$ if and only if $\omega_{R/\Ic}$
is $R/\Ic$-free, or equivalently, $R/\Ic$ is a
Gorenstein ring. \qed
\end{proof}

\medskip

Taking Remark \ref{obs1.def.urich.ideal} into account, the corollary below is readily seen to generalize \cite[Corollary 5.2]{goto2014}.

\begin{corollary}\label{cor.teo5.2.goto.generalizado}
Let $R$ be a Cohen-Macaulay local ring with canonical module
$\omega_R$, and let $M$ be a maximal Cohen-Macaulay $R$-module. Assume that the ideal $\Ic$ is Gorenstein. Then the following
assertions are equivalent:
\begin{itemize}
\item[(i)] $M$ is an Ulrich $R$-module with respect to $\Ic$;
\item[(ii)] $\textrm{\emph{Hom}}_{R}(M, \omega_R)$ is an Ulrich $R$-module with respect to $\Ic$.
\end{itemize}
\end{corollary}
\begin{proof} We have an isomorphism $M \cong
    \textrm{Hom}_{R} (\textrm{Hom}_{R} (M, \omega_R), \omega_R)$ (see Lemma \ref{teo.canonical.module}(iii)). Now the
    conclusion follows by Corollary
    \ref{cor1.teo5.1.goto.generalizado}. \qed
\end{proof}

\medskip

Our next result is a far-reaching extension of \cite[Lemma 2.2]{BHU} (see also Corollary \ref{regularity}).

\begin{corollary}\label{far}
Let $R$ be a Cohen-Macaulay local ring with canonical module
$\omega_R$. Assume that the ideal $\Ic$ is Gorenstein. Then the following assertions
are equivalent:
\begin{itemize}
\item[(i)] $\Ic$ is a parameter ideal;
\item[(ii)] $R$ is an Ulrich $R$-module with respect to $\Ic$;
\item[(iii)] $\omega_R$ is an Ulrich
    $R$-module with respect to $\Ic$.
\end{itemize}
\end{corollary}
\begin{proof} The equivalence
(i) $\Leftrightarrow$ (ii) is immediate from Definition 
\ref{def.I.urich.module} and holds regardless of $\Ic$ being Gorenstein. Now, by virtue of the isomorphisms $\omega_R \cong \textrm{Hom}_{R}
(R, \omega_R)$ and $\textrm{Hom}_{R}
(\omega_R, \omega_R)\cong R$, our Corollary
\ref{cor.teo5.2.goto.generalizado} yields (ii) $\Leftrightarrow$ (iii).\qed
\end{proof}

\medskip

As yet another byproduct of Theorem
\ref{teo.teo5.1.goto.generalizado}, we retrieve 
\cite[Proposition 4.1]{kobayashi2018}, which in turn generalizes the local version of
\cite[Proposition 3.5]{ulrich2003}.

\begin{corollary}\label{cor2.teo5.1.goto.generalizado}
{\rm (\cite[Proposition 4.1]{kobayashi2018})} Let $R$ be a Cohen-Macaulay
local ring of dimension $d$. Let $M,N$ be maximal Cohen-Macaulay
$R$-modules such that $\textrm{\emph{Hom}}_{R} (M, N) \neq 0$ and $\textrm{\emph{Ext}}_{R}^i(M, N)=0$ for all $i=1, \ldots, d-1$. If either $M$ or $N$ is an Ulrich $R$-module, then
so is $\textrm{\emph{Hom}}_{R} (M, N)$.
\end{corollary}

\begin{proof}
As observed in Remark \ref{obs.I.urich.module}(i), $M$ is an Ulrich $R$-module if
and only if $M$ is an Ulrich $R$-module with respect to the maximal ideal
$\Mc$ of $R$. Now, being a (finite-dimensional) vector space over the residue field $k=R/\Mc$, the module 
$\textrm{Hom}_{R}(M, N)/{\Mc} \textrm{Hom}_{R} (M, N)$
is $k$-free. Thus, $\textrm{Hom}_{R} (M, N)$ is Ulrich by Theorem
\ref{teo.teo5.1.goto.generalizado}(a). \qed
\end{proof}

\subsection{Hom with values in a semidualizing module} Let us recall that a finite module $\Cc$ over a ring $R$ is called \emph{semidualizing} if the morphism $R \rightarrow \textrm{Hom}_{R}(\Cc, \Cc)$ given by homothety is an
isomorphism and $\textrm{Ext}_{R}^i (\Cc, \Cc) = 0$ for all $i > 0$. In this case, a finite $R$-module $M$ is said to be
\emph{totally $\Cc$-reflexive} if the biduality map $M \rightarrow
\textrm{Hom}_{R} (\textrm{Hom}_{R} (M, \Cc), \Cc)$ is an isomorphism
and, in addition, $\textrm{Ext}_{R}^i (M, \Cc) = 0 = \textrm{Ext}_{R}^i
(\textrm{Hom}_{R} (M, \Cc), \Cc)$ for all $i
> 0$. Note every totally $\Cc$-reflexive module is maximal Cohen-Macaulay by virtue of the relative Auslander-Bridger formula (see \cite[Proposition 6.4.2]{sather}). A detailed account about the theory of semidualizing modules is given in \cite{sather}.

As a matter of illustration, $R$ is semidualizing as a module over itself, and, for any semidualizing $R$-module $\Cc$, both $R$ and $\Cc$ are totally $\Cc$-reflexive. More interestingly, if $R$ is a Cohen-Macaulay local ring possessing a canonical module $\omega_R$, then $\omega_R$ is semidualizing and, in addition, every maximal Cohen-Macaulay $R$-module is totally $\omega_R$-reflexive (to see this, use Lemma \ref{teo.canonical.module}). It should also be pointed out, based on the existence of several examples in the literature, that not every semidualizing $R$-module must be isomorphic to $R$ or $\omega_R$; see, e.g., \cite[5]{A-I} and \cite[2.3]{sather}.

\begin{corollary}\label{cor3.teo5.1.goto.generalizado}
Let $R$ be a Cohen-Macaulay local ring with a semidualizing
module $\Cc$, and let $M$ be a totally $\Cc$-reflexive $R$-module. Then,
$M$ is an Ulrich $R$-module if and only if
$\textrm{\emph{Hom}}_{R} (M, \Cc)$ is an Ulrich $R$-module.
\end{corollary}
\begin{proof} We have $M \cong \textrm{Hom}_{R} (\textrm{Hom}_{R} (M, \Cc),
\Cc)$, which in particular forces the module $\textrm{Hom}_{R}(M, \Cc)$ to be non-trivial, and in addition $$\textrm{Ext}_{R}^i(M,
\Cc)=0=\textrm{Ext}_{R}^i(\textrm{Hom}_{R} (M, \Cc), \Cc) \quad
\mbox{for \,all} \quad i > 0.$$  Since $\Cc$ is semidualizing, $\textrm{depth}_R\Cc
= \textrm{depth}\,R$ (see, e.g., \cite[Theorem 2.2.6(c)]{sather}) and hence $\Cc$ is
maximal Cohen-Macaulay. Now the result is clear by  Corollary
\ref{cor2.teo5.1.goto.generalizado}. \qed
\end{proof}

\medskip

Note that Corollary \ref{cor3.teo5.1.goto.generalizado} gives a different proof of the case $\Ic = \Mc$ of Corollary
\ref{cor.teo5.2.goto.generalizado} by taking $\Cc = \omega_R$. Another byproduct of Corollary
\ref{cor3.teo5.1.goto.generalizado} is the following curious characterization of regular local rings.

\begin{corollary}\label{regularity} Let $R$ be a Cohen-Macaulay local ring with a semidualizing
module $\Cc$. Then, $R$ is regular if and only if $\Cc$ is an Ulrich $R$-module.
\end{corollary}
\begin{proof} According to \cite[Proposition 2.1.12]{sather}, saying that $\Cc$ is semidualizing is tantamount to $R$ being a totally $\Cc$-reflexive $R$-module. Now, Corollary \ref{cor3.teo5.1.goto.generalizado} yields that $R$ is Ulrich over itself if and only if $\Cc$ is an Ulrich $R$-module. The former situation, as observed in \cite[Lemma 2.2]{BHU}, is equivalent to the regularity of $R$. \qed
 \end{proof}

\medskip

We raise the following question and a related remark.

\begin{question}\label{ques}\rm Does Corollary \ref{cor1.teo5.1.goto.generalizado} hold with $\Cc$ (a given semidualizing $R$-module) in place of $\omega_R$?

\end{question}

\begin{remark}\rm An affirmative answer to Question \ref{ques} would imply the validity of Corollary \ref{cor.teo5.2.goto.generalizado} with $\Cc$ in place of $\omega_R$ as well, provided that $R$ is a normal domain. Indeed, it suffices to note that in this case the maximal Cohen-Macaulay $R$-module $M$ is necessarily reflexive in the usual sense, and thus by \cite[Corollary 5.4.7]{sather} (which also requires $R$ to be normal) we have $$M \cong \textrm{Hom}_{R}(\textrm{Hom}_{R}(M, \Cc), \Cc)$$ via the natural biduality map.

\end{remark}

\subsection{Freeness criteria for $M/\Ic M$ via (co)homology vanishing} We close the section providing some criteria for the freeness of the $R/\Ic$-module $M/\Ic M$, which is of interest since this is one of the requirements for $M$ to be Ulrich with respect to $\Ic$ (see Definition \ref{def.I.urich.module}).

As we have been investigating how Ulrichness behaves under the Hom ($={\rm Ext}^0$) functor, it seems natural to wonder about the relevance of higher Ext modules in the theory, and in fact we shall see that the vanishing of finitely many ``diagonal" Ext modules ${\rm Ext}_{R/\Ic}^i(M/\Ic M, M/\Ic M)$, under suitable hypotheses, can detect freeness over the Artinian local ring $R/\Ic$ (which we will assume to be Gorenstein). Vanishing of homology modules, namely ``diagonal" Tor modules ${\rm Tor}^{R/\Ic}_j(M/\Ic M, M/\Ic M)$, will also play a role. Essentially, our criteria will consist of adaptations of some results from \cite{HSV} and one from \cite{Sega}.

In the proposition below, and as before, $(R, \Mc)$ and $\Ic$ (also $Q$, which appears in the proof) are as in Convention \ref{convention}, and $\ell_R(-)$ stands for length of $R$-modules.

\begin{proposition} Suppose $R/\Ic$ is Gorenstein {\rm (}e.g., if $R$ is Gorenstein and $\Ic$ is Ulrich; see Remark \ref{obs1.def.urich.ideal}{\rm )} and let $M$ be a finite $R$-module. Assume any one of the following situations:
\begin{itemize}
\item[(i)] ${\Mc}^2M\subset \Ic M$ and ${\rm Ext}_{R/\Ic}^i(M/\Ic M, M/\Ic M)=0$ for all $i$ satisfying $1\leq i\leq {\rm max}\{3, \nu(M), \ell_R(M/\Ic M)-\nu(M)\}$;
\item[(ii)] ${\Mc}^3\subset \Ic$ and ${\rm Ext}_{R/\Ic}^i(M/\Ic M, M/\Ic M)=0$ for some $i>0$;
\item[(iii)]  {\rm (}$R/\Ic$ not necessarily Gorenstein.{\rm )} $R/\Mc$ is infinite, $\Ic$ is not a parameter ideal, ${\Mc}^3\subset \Ic$,\, ${\rm e}_{\Ic}^0(R)\leq 2\ell_R(\Mc /({\Mc}^2+\Ic))$, and ${\rm Tor}^{R/\Ic}_j(M/\Ic M, M/\Ic M)=0$ for three consecutive values of $j\geq 2$;
\item[(iv)] ${\Mc}^4\subset \Ic$, there exists $x\in \Mc \setminus \Ic$ such that the ideal $(\Ic \colon x)/\Ic$ is principal, and ${\rm Tor}^{R/\Ic}_j(M/\Ic M, M/\Ic M)=0$ for all $j\gg 0$.
\end{itemize}
Then,  $M/\Ic M$ is $R/\Ic$-free.
\end{proposition}
\begin{proof} For simplicity, set $\overline{R}=R/\Ic$, $\overline{\Mc}=\Mc/\Ic$, and $\overline{M}=M/\Ic M$. Let us assume (i). By assumption ${\overline{\Mc}}^2{\overline{M}}=0$, hence $$\nu({\overline{\Mc}}\,{\overline{M}})=\ell_R({\overline{\Mc}}\,{\overline{M}})=\ell_R(\Mc M/\Ic M).$$ On the other hand, by the short exact sequence $$0 \longrightarrow \Mc M/\Ic M \longrightarrow M/\Ic M \longrightarrow M/\Mc M \longrightarrow 0$$ we have $\ell_R(\Mc M/\Ic M)=\ell_R(M/\Ic M)-\ell_R(M/\Mc M)$. Therefore we obtain $\nu({\overline{\Mc}}\,{\overline{M}})=\ell_R(M/\Ic M)-\nu(M)$. In addition it is clear that $\nu(\overline{M})=\nu(M)$. Now we can apply  \cite[Proposition 4.4(1)]{HSV}, which ensures that the $R/\Ic$-module $M/\Ic M$ is either free or injective. Since $R/\Ic$ is Gorenstein, $M/\Ic M$ is necessarily free, as needed.

Assume that (ii) holds. Notice that ${\overline{\Mc}}^3=0$ by hypothesis. Now, since $R/\Ic$ is Gorenstein, the freeness of $M/\Ic M$ follows readily by \cite[Theorem 4.1(2)]{HSV}. 

Now suppose (iii). Let $\ell \ell (\overline{R})$ denote the {\it Loewy length} of $\overline{R}$, which is the smallest integer $n$ such that $\overline{\Mc}^n=0$, i.e., ${\Mc}^n\subset \Ic$. Thus, by assumption, $\ell \ell (\overline{R})\leq 3$. If $\ell \ell (\overline{R})=1$ (i.e., $\Ic =\Mc$), there is nothing to prove. If $\ell \ell (\overline{R})=2$, then $M/\Ic M$ is free by \cite[Remark 2.1]{HSV}. So we can assume $\ell \ell (\overline{R})=3$. Using Remark \ref{obs.I.urich.module}(i) and the hypothesis that $\Ic$ is not a parameter ideal (so that the inclusion $Q\subset \Ic$ is strict), we get
${\rm e}^0_{\Ic}(R) = \ell_R(R/Q) \geq \ell_R(R/\Ic) + 1$. Therefore, we can write $$2\nu(\overline{\Mc})=2\ell_R(\Mc /({\Mc}^2+\Ic))\geq {\rm e}^0_{\Ic}(R)\geq \ell_R(\overline{R}) + 1=\ell_R(\overline{R}) - \ell \ell (\overline{R}) + 4.$$ Now we are in a position to apply \cite[Theorem 3.1(2)]{HSV} to conclude that $M/\Ic M$ is free.

Finally, suppose (iv). So $R/\Ic$ is Gorenstein and $\overline{\Mc}^4=0$, and in addition note that $(\Ic \colon x)/\Ic$ is the annihilator of $x\overline{R}$. Then  $M/\Ic M$ is free by \cite[Theorem 3.3]{Sega}. \qed
\end{proof}

\begin{remark}\rm From the proof in the situation (iii) it is clear that, for general $\Ic$ (i.e., possibly a parameter ideal), the hypothesis on the multiplicity must be replaced with ${\rm e}_{\Ic}^0(R)\leq 2\ell_R(\Mc /({\Mc}^2+\Ic))-1$. 

\end{remark}

\section{Horizontal linkage and the Ulrich property} \label{sec4}

We begin this section by pointing out the warming-up fact that, if the local ring $R$ is Gorenstein, then it follows from \cite[Theorem
1]{link2004} that every stable Ulrich $R$-module with respect to $\Ic$ -- where $\Ic$ is as in Convention \ref{convention} -- is horizontally linked (note that maximal Cohen-Macaulay modules are precisely the totally reflexive modules, since
$R$ is Gorenstein). See Subsection \ref{linkage-subsect} for terminology. 

In essence, our goal herein is to develop a further study of linkage of Ulrich modules with respect to $\Ic$ (also assumed to be Ulrich but not a parameter ideal), the main result being the theorem below, which in particular shows that the operation of horizontal linkage over a
Gorenstein local ring preserves the Ulrich property with respect to
$\Ic$ for horizontally linked modules.

\begin{theorem}\label{teo.link.ulrich}
Let $(R, \Mc)$ be a Cohen-Macaulay local ring of
dimension $d$, and suppose the ideal $\Ic$ is Ulrich but not a parameter ideal. Consider the following
assertions:
\begin{itemize}
\item[(i)] $R$ is Gorenstein;
\item[(ii)] $M$ is Ulrich with respect to $\Ic$ if and only if $\lambda M$ is Ulrich with
respect to $\Ic$, whenever $M$ is a horizontally linked $R$-module;
\item[(iii)] $\lambda M$ is maximal Cohen-Macaulay, whenever $M$ is a horizontally linked $R$-module which is Ulrich with respect to $\Ic$;
\item[(iv)] $\textrm{\emph{Ext}}_{R}^{d+2}(R/\Ic,R)=0$.
\end{itemize}
Then the following statements hold:
\begin{enumerate}
\item[(a)] {\rm (i)} $\Rightarrow$ {\rm (ii)} $\Rightarrow$ {\rm (iii)};
\item[(b)] If $d\geq 2$, then {\rm (iii)} $\Rightarrow$ {\rm (iv)};
\item[(c)] If $d\geq 2$ and $\Ic = \Mc$, then all the four conditions above are equivalent.
\end{enumerate}
\end{theorem}
\begin{proof}
(a) (i) $\Rightarrow$ (ii). Let $M$ be a
horizontally linked $R$-module. By Lemma \ref{teo2.link}, $M$ is
a stable $R$-module. Assume that $M$ is an Ulrich $R$-module with respect to $\Ic$. By \cite[Corollary 5.3]{goto2014},
the Auslander transpose $\textrm{Tr} M$ is Ulrich with respect to $\Ic$. Moreover, since $M$ is stable, we obtain by \cite[Theorem
32.13]{fuller} that
$\textrm{Tr} M$ is stable as well. Applying \cite[Corollary 5.3]{goto2014} we conclude that the syzygy module $\Omega \textrm{Tr} M=\lambda M$ is Ulrich with respect to $\Ic$. Now, to see the converse, it suffices to apply Lemma \ref{prop4.link} to the module $\lambda M$ and to use that $M\cong \lambda^2 M$. Notice that (ii) $\Rightarrow$ (iii) is obvious. This concludes the proof of (a).

\medskip

(b) (iii) $\Rightarrow$ (iv). Let $\overline{R} = R/\Ic$, and assume contrarily that $\textrm{Ext}_{R}^{d+2}(\overline{R}, R)\neq 0$. First notice that $\Omega^{d+1}\overline{R}$ is stable, otherwise $R$ would be a direct summand of $\Omega^{d+1}\overline{R}$ and then, by \cite[Corollary
1.2.5]{avramov96}, $$d+1 \leq \max \{ 0, \textrm{depth}\,R - \textrm{depth}_R\overline{R} \} = d - \textrm{depth}_R\overline{R},$$
which is absurd. Now, by Lemma \ref{teo2.link},
$\Omega^{d+1}\overline{R}$ is a horizontally linked $R$-module. By
\cite[Theorem 3.2]{goto2014}, $\Omega^{d+1}\overline{R}$ is an
Ulrich $R$-module with respect to $\Ic$. It follows from
the assumption of (iii) that $\lambda\Omega^{d+1}\overline{R}$ is a maximal
Cohen-Macaulay $R$-module, which in turn fits into a short exact
sequence
$$0 \longrightarrow \lambda\Omega^{d+1}\overline{R} \longrightarrow
F \longrightarrow  \textrm{Tr} \Omega^{d+1}\overline{R} \longrightarrow
0$$ for some free $R$-module $F$. By \cite[Proposition
1.2.9]{CMr}, we get
\begin{equation}\label{positive-depth}
\textrm{depth}_R\textrm{Tr} \Omega^{d+1}\overline{R} \geq
\min \{ \textrm{depth}_RF, \textrm{depth}_R\lambda\Omega^{d+1}\overline{R} - 1 \} = d-1 > 0.
\end{equation}
Using (\ref{eq.seq.exact}), there is an exact sequence $$0
\rightarrow \textrm{Ext}_{R}^{1}(\textrm{Tr}\textrm{Tr}
\Omega^{d+1}\overline{R}, R) \rightarrow \textrm{Tr}
\Omega^{d+1}\overline{R} \rightarrow (\textrm{Tr}
\Omega^{d+1}\overline{R})^{**} \rightarrow
\textrm{Ext}_{R}^{2}(\textrm{Tr}\textrm{Tr}
\Omega^{d+1}\overline{R}, R) \rightarrow 0$$ and since
$\Omega^{d+1}\overline{R}$ is stable, we have $\textrm{Tr}\textrm{Tr}
\Omega^{d+1}\overline{R} \cong \Omega^{d+1}\overline{R}$ by
\cite[Corollary 32.14(4)]{fuller}. Thus, we obtain the exact sequence
\begin{equation}\label{eq2.teo.link.ulrich}
0 \longrightarrow \textrm{Ext}_{R}^{d+2}(\overline{R}, R) \longrightarrow
\textrm{Tr} \Omega^{d+1}\overline{R} \longrightarrow (\textrm{Tr}
\Omega^{d+1}\overline{R})^{**} \longrightarrow
\textrm{Ext}_{R}^{d+3}(\overline{R}, R) \longrightarrow 0.
\end{equation}
As $\Ic$ is $\Mc$-primary, the non-zero module $\textrm{Ext}_{R}^{d+2}(\overline{R}, R)$ must have finite length, which in particular implies $\textrm{depth}_R\textrm{Ext}_{R}^{d+2}(\overline{R}, R) = 0$. On the other hand, by virtue of (\ref{positive-depth}) and
(\ref{eq2.teo.link.ulrich}), we get  $\textrm{depth}_R\textrm{Ext}_{R}^{d+2}(\overline{R}, R)>0$, a contradiction.

\medskip

(c) (iv) $\Rightarrow$ (i). If $\textrm{Ext}_{R}^{d+2}(R/\Mc, R)=0$ then, by \cite[Theorem
18.1]{matsumura}, the local ring $R$ is Gorenstein. \qed
\end{proof}

\medskip

In order to provide the first application of our theorem, we invoke the following classical concept.

\begin{definition}\label{min-mult}\rm A $d$-dimensional Cohen-Macaulay local ring $R$ is said to have \textit{minimal multiplicity} if its multiplicity and embedding dimension are related by
$\textrm{e}(R) = \textrm{edim}\,R - d+ 1$. As is well-known, there is in general an inequality $\textrm{e}(R) \geq \textrm{edim}\,R - d+ 1$, which originates the terminology.
\end{definition}

Now recall that a local ring $R$ is a {\it hypersurface} ring if $R\cong S/(f)$, where $(S, \Nc)$ is a regular local ring and $f\in {\Nc}$. Such a ring is said to be a {\it quadratic} hypersurface ring if $f\in {\Nc}^2\setminus {\Nc}^3$.
Clearly, a hypersurface ring $R\cong S/(f)$ with $f\in {\Nc}^2$ is quadratic if and only if $R$ has minimal multiplicity (equal to 2). 

Our Theorem \ref{teo.link.ulrich} yields a characterization of quadratic hypersurface rings in terms of linkage of Ulrich modules in the classical sense, i.e., in the case $\Ic=\Mc$. It is worth recalling an interesting connection (which we shall use in the proof of Corollary \ref{cor.I.Ulrich.preservada.por.LH}) between quadratic hypersurface rings and the Ulrich property, to wit, every non-free maximal Cohen-Macaulay module over such a ring is a direct sum of an Ulrich module and a free module (see \cite[Corollary 1.4]{H-K}); in particular, any such ring admits an Ulrich module.

\begin{corollary}\label{Hyper-charact} Let $R$ be a non-regular Cohen-Macaulay local ring of minimal multiplicity with dimension $d\geq 2$ and infinite residue field $k$. The following assertions are equivalent:
\begin{itemize}
\item[(i)] $R$ is a $($quadratic$)$ hypersurface ring;
\item[(ii)] $M$ is Ulrich if and only if $\lambda M$ is Ulrich, whenever $M$ is a horizontally linked $R$-module;
\item[(iii)] $\lambda M$ is maximal Cohen-Macaulay, whenever $M$ is a horizontally linked Ulrich $R$-module;
\item[(iv)] $\textrm{\emph{Ext}}_{R}^{d+2}(k, R)=0$.
\end{itemize}
\end{corollary}
\begin{proof} As before let $\Mc$ be the maximal ideal of $R$. Since $R/\Mc$ is infinite, it is well-known that $R$ has minimal multiplicity if and only if $${\Mc}^2 = ({\bf x}){\Mc}$$ with ${\bf x}$ an $R$-sequence (see \cite[Exercise 4.6.14]{CMr}), which in turn means that $\Mc$ is an Ulrich ideal in the sense of Definition \ref{def.urich.ideal}. Since $R$ is non-regular, $\Mc$ is not a parameter ideal. Therefore, as every hypersurface ring is Gorenstein, the implications {\rm (i)} $\Rightarrow$ {\rm (ii)} $\Rightarrow$ {\rm (iii)} $\Rightarrow$ {\rm (iv)} follow readily by Theorem \ref{teo.link.ulrich} with $\Ic=\Mc$. Now, as recalled in the proof of the theorem, condition (iv) forces $R$ to be Gorenstein. But it is well-known that a Gorenstein local ring having minimal multiplicity is just a quadratic hypersurface ring, as needed. \qed
\end{proof}

\medskip

Connections between a more general notion of minimal multiplicity
and the Ulrich property with respect to $\Ic$ will be given in Section \ref{sec5}.

Before establishing another consequence of Theorem \ref{teo.link.ulrich} (over Gorenstein local rings), we invoke an auxiliary invariant which will be used in the proof, namely, the {\it Gorenstein dimension} of a finite $R$-module $M$, which is denoted G-$\textrm{dim}_RM$ (for the definition, see, e.g., \cite[Definition 1.2.3]{GD}). Recall that if $R$ is Gorenstein then G-$\textrm{dim}_RM<\infty$ for every finite $R$-module $M$. If $R$ is local and $M$ is a finite $R$-module with G-$\textrm{dim}_RM<\infty$ then the so-called Auslander-Bridger formula states that G-$\textrm{dim}_RM={\rm depth}\,R-{\rm depth}_RM$. In particular, if $R$ is Gorenstein, then G-$\textrm{dim}_RM=0$ if and only if $M$ is maximal Cohen-Macaulay. For details, see \cite{AB}, also \cite{GD}.

\begin{corollary}\label{prop.I.Ulrich.preservada.por.LH}
Let $R$ be a Gorenstein local ring, and suppose the ideal $\Ic$ is
Ulrich but not a parameter ideal. Let $M$ be
a stable maximal Cohen-Macaulay $R$-module. Then, $M$ is an Ulrich
$R$-module with respect to $\Ic$ if and only if $\lambda
M$ is an Ulrich $R$-module with respect to $\Ic$.
\end{corollary}
\begin{proof}
Since $R$ is Gorenstein and $M$ is maximal Cohen-Macaulay, then as observed above we have  G-$\textrm{dim}_RM=0$. By \cite[Theorem 1]{link2004}, $M$ is horizontally linked.
Now the result follows from Theorem \ref{teo.link.ulrich}(a). \qed
\end{proof}

\begin{corollary}\label{cor.I.Ulrich.preservada.por.LH} Let $R$ be quadratic hypersurface local ring with infinite residue field, and let $M$ be a stable maximal Cohen-Macaulay $R$-module. Then, $\lambda M$ is an Ulrich $R$-module.
\end{corollary}
\begin{proof} Over such a ring, any maximal Cohen-Macaulay module $M$ is either free or satisfies $$M\cong U\oplus F,$$ for some Ulrich module $U$ and free module $F$, according to \cite[Corollary 1.4]{H-K}. Thus, if in addition $M$ is stable (in particular, non-free), then it must be Ulrich. Also note the maximal ideal $\Mc$ of $R$ is Ulrich but not a parameter ideal. Now we can apply Corollary \ref{prop.I.Ulrich.preservada.por.LH} with $\Ic = \Mc$ to get the result. \qed \end{proof}

\medskip

Before giving more consequences of Corollary \ref{prop.I.Ulrich.preservada.por.LH}, we recall a useful lemma.

\begin{lemma}{\rm (\cite[Lemma 1.2]{H-K})}\label{stable} Let $R$ be a Gorenstein local ring. If $M$ is a maximal Cohen-Macaulay $R$-module, then $\Omega M$ is a stable $R$-module. 
\end{lemma}

\begin{corollary}\label{lambda-syz} Let $R$ be a Gorenstein local ring of dimension $d$, and suppose the ideal $\Ic$ is Ulrich but not a parameter ideal. Then, $\lambda (\Omega^k \Ic)$ is an Ulrich $R$-module with respect to $\Ic$ for all $k\geq d$.
\end{corollary}
\begin{proof} First, as recalled in Remark \ref{obs.I.urich.module}(iii), the $R$-module $\Omega^k(R/\Ic)$ is Ulrich with respect to $\Ic$ (in particular, maximal Cohen-Macaulay) for all $k\geq d$. It follows by Lemma \ref{stable} that the $R$-module $\Omega^{k+1}(R/\Ic)=\Omega^{k}\Ic$ is stable for all $k\geq d$, and thus Corollary \ref{prop.I.Ulrich.preservada.por.LH} concludes the proof. \qed
\end{proof}

\begin{corollary}\label{linkage-Ulr-ideal} Let $R$ be a 1-dimensional Gorenstein local ring. If $\Ic$ is an Ulrich ideal of $R$ which is not a parameter ideal, then $\lambda {\Ic}$ is an Ulrich $R$-module with respect to $\Ic$.
\end{corollary} 
\begin{proof} By Remark \ref{obs.I.urich.module}(ii), $\Ic$ is an Ulrich $R$-module with respect to $\Ic$. Note $\Ic$ is stable as it is a non-principal ideal, hence a non-free $R$-module.  Now, apply Corollary \ref{prop.I.Ulrich.preservada.por.LH}. \qed
\end{proof}

\section{Minimal multiplicity and Ulrich properties}\label{sec5}

We start the section presenting a number of preparatory definitions (e.g., Rees and associated graded modules, and relative reduction numbers) as well as some auxiliary facts.

Let $I$ be a proper ideal of a ring $R$. Recall that the  Rees algebra of $I$ is  the graded ring ${\mathcal R}(I)  =  \bigoplus_{n \geq 0}I^{n}$ (as usual, we put $I^0=R$), which can be realized as the standard graded subalgebra $R[Iu]\subset R[u]$, where $u$ is an indeterminate over $R$. The associated graded ring of $I$ is given by 
${\mathcal G}(I)  = \bigoplus_{n \geq
0}I^{n}/I^{n + 1}={\mathcal R}(I)\otimes_RR/I$, which is standard graded over $R/I$.

\begin{definition}\rm If $M$ is a finite $R$-module, the {\it Rees module} and the {\it associated graded module} of $I$ relative to $M$ are, respectively, given by $${\mathcal R}(I, M)  =   \bigoplus_{n \geq 0} I^{n}M, \quad {\mathcal G}(I, M)  = \bigoplus_{n \geq 0} \frac{I^{n}M}{I^{n + 1}M}={\mathcal R}(I, M)\otimes_RR/I,$$ which are  
finite graded modules over ${\mathcal R}(I)$ and  ${\mathcal G}(I)$, respectively.
\end{definition}

Now consider a local ring $(R, \Mc)$ with residue field $k$. For a proper ideal $I$ of $R$, recall that the  fiber cone of $I$ is the special fiber ring of ${\mathcal R}(I)$, i.e., the standard graded $k$-algebra ${\mathcal F}(I)=\bigoplus_{n \geq
0}I^{n}/\Mc I^{n}={\mathcal R}(I)\otimes_Rk$. We can also consider the finite graded ${\mathcal F}(I)$-module ${\mathcal F}(I, M)  = \bigoplus_{n \geq 0} I^{n}M/\Mc I^{n}M={\mathcal R}(I, M)\otimes_Rk,$ whose Krull dimension (called {\it analytic spread} of $I$ relative to $M$) is denoted $${\rm s}_M(I) = {\rm dim}\,{\mathcal F}(I, M).$$

\begin{definition}\rm
Let $I$ be a proper ideal of a ring $R$ and let $M$ be a non-zero finite $R$-module. An ideal $J \subset I$ is called an \textit{$M$-reduction of $I$} if $JI^{n}M = I^{n + 1}M$ for some integer $n \geq 0$. Such an $M$-reduction $J$ is said to be \textit{minimal} if it is minimal with respect to inclusion. If $J$ is an $M$-reduction of $I$, we define the \textit{reduction number of $I$ with respect to $J$ relative to $M$} as
$${\rm r}_{J}(I, M) = \mathrm{min}\{m \in \mathbb{N} \mid JI^{m}M = I^{m + 1}M\}.$$
\end{definition}

The lemma below detects a useful connection between minimal $M$-reductions and the so-called ({\it maximal}) {\it $M$-superficial sequences} of a given $\Mc$-primary ideal in a local ring $(R, \Mc)$; for the definition and details about the latter concept, we refer to  \cite[1.2 and 1.3]{Rossi-Valla} (cf.\,also \cite{Conti}).

\begin{lemma}\label{connection} {\rm (\cite[Corollario 3.14]{Conti})} Let $(R, \Mc)$ be a local ring with infinite residue field and let $I$ be an $\Mc$-primary ideal. Let $M$ be a finite $R$-module of positive dimension. Then, every minimal $M$-reduction of $I$ can be generated by a maximal $M$-superficial sequence of $I$. Conversely, an ideal generated by a maximal $M$-superficial sequence of $I$ is necessarily a minimal $M$-reduction of $I$.
\end{lemma}

Next we invoke a central notion in this section, and a helpful lemma. As in Subsection \ref{ulrich-subsect}, if $I$ is an ideal of definition of a finite $R$-module $M$ then $\textrm{e}_I^0(M)$ denotes the multiplicity of $M$ with respect to $I$. Moreover, we let $\textrm{e}_I^1(M)$ stand for the first Hilbert coefficient -- the so-called {\it Chern number} -- of $M$ with respect to $I$.

\begin{definition}{\rm (\cite[Definition 15]{Puthenpurakal2})} \rm Let $(R, \Mc)$ be a
local ring, $M$ a Cohen-Macaulay $R$-module of dimension $t$ and
$I$ a proper ideal of $R$ such that ${\Mc}^{n}M \subset IM$ for
some $n > 0$. Then $M$ has \emph{minimal multiplicity with
respect to $I$} if $$\textrm{e}_I^0 (M) = (1 - t)\ell_R(M/IM) +
\ell_R(IM/I^2M).$$
\end{definition}

Notice that by taking $M=R$ and $I=\Mc$ we recover Definition \ref{min-mult}.

\begin{lemma}\label{Puthenchar} {\rm (\cite[Theorem 16]{Puthenpurakal2})} Let $(R, \Mc)$ be a local ring,
$M$ a Cohen-Macaulay $R$-module of dimension $t$ and $I$ a proper ideal
of $R$ such that ${\Mc}^{n}M \subset IM$ for some $n >
0$. The following conditions are equivalent:
\begin{itemize}
\item[(i)] $M$ has minimal multiplicity with respect to $I$;
\item[(ii)] $(z_{1}, \ldots,z_{t})IM = I^2M$, for every maximal $M$-superficial sequence $z_{1}, \ldots,z_{t}$;
\item[(iii)] $(z_{1}, \ldots,z_{t})IM = I^2M$, for some maximal $M$-superficial sequence $z_{1}, \ldots,z_{t}$;
\item[(iv)] ${\rm e}_{I}^{1}(M) = {\rm e}_{I}^{0}(M) - \ell_R(M/IM).$
\end{itemize}
\end{lemma}

\smallskip

Here we observe that item (iii) above is not present in \cite{Puthenpurakal2}, but a simple inspection of the proof easily shows that this assertion is also equivalent to the ones given in   \cite[Theorem 16]{Puthenpurakal2}.

Our first result in this part is the following. As in the previous sections, we let $Q=(x_1, \ldots, x_d)\subset \Ic$ be as in Convention \ref{convention}.

\begin{proposition}\label{UlrichMultMin}
Suppose $R$ is a Cohen-Macaulay local ring with infinite residue field. Then, every Ulrich $R$-module with respect to $\Ic$ has minimal multiplicity with respect to $\Ic$.
\end{proposition}
\begin{proof} Let $M$ be an Ulrich module with respect to $\Ic$. In particular, $M$ is maximal Cohen-Macaulay. Let $\mathrm{grade}(\Ic, M)$ denote the maximal length of an $M$-sequence contained in $\Ic$. By \cite[Lemma 1.3 and Lemma 1.6]{Kadu}, we have
$$\mathrm{grade}(\Ic, M) \leq {\rm s}_{M}(\Ic) \leq \mathrm{dim}\,M.$$ As $\Ic$ is  $\Mc$-primary, $\mathrm{grade}(\Ic, M) ={\rm depth}\,M=d$, where as before $d={\rm dim}\,R$. Hence ${\rm s}_{M}(\Ic)=d=\nu(Q)$, where $\nu(-)$ stands for minimal number of generators. As is well-known (see, e.g., \cite[Corollario 3.22]{Conti}), this implies that $Q$ is a minimal $M$-reduction of $\Ic$, and therefore Lemma \ref{connection} gives that $x_1, \ldots, x_d$ is in fact a maximal $M$-superficial sequence of $\Ic$. On the other hand, because $M$ is Ulrich, we have $QM=\Ic M$ and so $$Q\Ic M={\Ic}^2M.$$
We conclude, by Lemma \ref{Puthenchar}, that $M$ has minimal multiplicity with respect to $\Ic$. \qed
\end{proof}

\begin{remark}\rm The converse of Proposition \ref{UlrichMultMin} fails even in the classical case $\Ic = \Mc$ (see \cite[Example 4.12]{Puthenpurakal3}).

\end{remark}

Combining Proposition \ref{UlrichMultMin}  and \cite[Theorem 16]{Puthenpurakal2}, we immediately obtain the following property.

\begin{corollary}\label{CM-G} Suppose $R$ is a Cohen-Macaulay local ring with infinite residue field. If $M$ is an Ulrich $R$-module with respect to $\Ic$, then the associated graded ${\mathcal G}(\Ic)$-module ${\mathcal G}(\Ic, M)$ is Cohen-Macaulay.
    
\end{corollary}

The next consequence deals with the Chern number and gives a generalization of \cite[Corollary 1.3(1)]{Ooishi}.

\begin{corollary}\label{UlrHilb}
Let $(R, \Mc)$ be a Cohen-Macaulay local ring with infinite residue field and positive dimension, and let $M$ be a maximal Cohen-Macaulay $R$-module. Then ${\rm e}^{1}_{\Ic}(M) \geq 0$, and the following assertions are equivalent:
\begin{itemize}
\item[(i)]  $M$ is an Ulrich $R$-module with respect to $\Ic$;
\item[(ii)] $M/\Ic M$ is a free $R/\Ic $-module and ${\rm e}^{1}_{\Ic}(M) = 0$.
\end{itemize}
\end{corollary}
\begin{proof}
Applying \cite[Proposition 12]{Puthenpurakal2} and Remark \ref{obs.I.urich.module}(i), we get $${\rm e}_{\Ic }^{1}(M) \geq {\rm e}_{\Ic }^{0}(M) - \ell_R(M/\Ic M)\geq 0.$$ If $M$ is Ulrich with respect to $\Ic $ then, by definition, the $R/\Ic $-module $M/\Ic M$ is free and in addition ${\rm e}_{\Ic }^{0}(M) = \ell_R(M/\Ic M)$ (use again Remark \ref{obs.I.urich.module}(i)). On the other hand, Proposition \ref{UlrichMultMin} ensures that $M$ has minimal multiplicity with respect to $\Ic $, and therefore Lemma \ref{Puthenchar} gives ${\rm e}_{\Ic }^{1}(M) = {\rm e}_{\Ic }^{0}(M) - \ell_R(M/\Ic M) = 0$.
	
	Conversely, suppose (ii). Since $M$ is already assumed to be maximal Cohen-Macaulay, it remains to show that $\Ic M=QM$, which as we know is equivalent to the equality $\textrm{e}_{\mathscr{I}}^0 (M) =
\ell_R(M/\mathscr{I}M)$. But this follows from
$0\leq {\rm e}_{\Ic }^{0}(M) - \ell_R(M/\Ic M) \leq {\rm e}_{\Ic }^{1}(M)  = 0$. This concludes the proof.\qed
\end{proof}

\begin{corollary}\label{Chern-of-syz} Let $(R, \Mc)$ be a Cohen-Macaulay local ring with infinite residue field and dimension $d\geq 1$. If $\Ic$ is an Ulrich ideal of $R$ which is not a parameter ideal, then ${\rm e}^{1}_{\Ic}(\Omega^{k}\Ic)  = 0$ for all $k\geq d-1$. If in addition $R$ is Gorenstein, then $${\rm e}^{1}_{\Ic}(\lambda (\Omega^{k}\Ic))  = 0 \quad \mbox{for \,all} \quad k\geq d.$$
    
\end{corollary}
\begin{proof} Recall that the $R$-module $\Omega^{k+1}(R/\Ic)=\Omega^{k}\Ic$ is Ulrich with respect to $\Ic$ (in particular, maximal Cohen-Macaulay) for all $k\geq d-1$; see Remark \ref{obs.I.urich.module}(iii). Then the vanishing of ${\rm e}^{1}_{\Ic}(\Omega^{k}\Ic)$ follows by Corollary \ref{UlrHilb}. Now if $R$ is Gorenstein then, by Corollary \ref{lambda-syz}, the module $\lambda (\Omega^k \Ic)$ is Ulrich with respect to $\Ic$ for all $k\geq d$, and we again apply Corollary \ref{UlrHilb}. \qed

\end{proof}

\begin{remarks}\label{Hilb-S}\rm (i) Let $M$ be a $d$-dimensional Cohen-Macaulay $R$-module (assume the setting of Convention \ref{convention}, with $d>0$ and $R/\Mc$ infinite). Recall that, for $k\gg 0$, the Hilbert-Samuel function  $H_{\Ic}^M(k)= \ell_R(M/{\Ic}^kM)$ coincides with a degree $d$ polynomial $P_{\Ic}^M(k)$ (the Hilbert-Samuel polynomial of $M$ with respect to $\Ic$) which can be expressed as 
$$P_{\Ic}^M(k) = \sum_{i = 0}^{d}(-1)^{i}{\rm e}^{i}_{\Ic}(M)\binom{k + d - i - 1}{d - i}.$$ Now if $M$ is Ulrich with respect to $\Ic$, then in particular $M/\Ic M\cong (R/\Ic)^{\nu(M)}$ and therefore, by Corollary \ref{UlrHilb}, we get ${\rm e}^{0}_{\Ic}(M)=\ell_R(M/\Ic M)=\nu(M)\ell_R(R/\Ic)$ and
${\rm e}^{1}_{\Ic}(M)=0$. Thus, if for instance $d=1$ then $P_{\Ic}^M(k) = \nu(M)\ell_R(R/\Ic) k$. If $d=2$, we have $$P_{\Ic}^M(k) = \nu(M)\ell_R(R/\Ic)\binom{k+1}{2}+ {\rm e}^{2}_{\Ic}(M),$$ which particularly raises the problem of finding ${\rm e}^{2}_{\Ic}(M)$. Of course, in case we know an integer $k_0$  satisfying $P_{\Ic}^M(k)=H_{\Ic}^M(k)$ for all $k\geq k_0$, then ${\rm e}^{2}_{\Ic}(M)$ can be computed from the expression above by evaluating $k=k_0$. 

\medskip

(ii) If $d\geq 1$ and $\Ic$ is an Ulrich ideal of $R$ then, as we know, the $j$-th syzygy module of $\Ic$ is Ulrich with respect to $\Ic$ for all $j\geq d-1$. Now assume $d=1$. Applying the preceding part to the module $\Omega^j\Ic$ for any $j\geq 0$, and noticing that $\nu(\Omega^j\Ic)$ is precisely the $j$-th Betti number $\beta_j(\Ic)$ of $\Ic$, we obtain the simple formula $$P_{\Ic}^{\Omega^j\Ic}(k) = \beta_j(\Ic)\ell_R(R/\Ic) k.$$ In addition, considering linkage and assuming that $R$ is Gorenstein,
our Corollary \ref{lambda-syz} yields that $\lambda(\Omega^j\Ic)$ is also Ulrich with respect to $\Ic$ for any $j\geq 1$, and observe that $\nu(\lambda(\Omega^j\Ic))=\beta_j(\Ic)$ as well. It follows that $P_{\Ic}^{\Omega^j\Ic}(k)=
P_{\Ic}^{\lambda(\Omega^j\Ic)}(k)$.

\end{remarks}

Our next result, Theorem \ref{regMulMin} below, provides a characterization of modules of minimal multiplicity in terms of reduction number and Castelnuovo-Mumford regularity (of blowup modules). For completeness, we recall the definition of the latter, which is of great importance in commutative algebra and algebraic geometry, for instance in the study of degrees of syzygies over polynomial rings; we refer, e.g., to \cite[Chapter 15]{B-S}.

Let $S = \bigoplus_{n \geq 0}S_{n}$ be a finitely generated standard graded algebra over a ring $S_{0}$. As usual, we write $S_{+} = \bigoplus_{n \geq 1}S_{n}$. For a graded $S$-module $A=\bigoplus_{n \in {\mathbb Z}}A_{n}$ satisfying $A_n=0$ for all $n\gg 0$,
we set
$${\rm end}\,A = \left\{ \begin{array}{lll}\textrm{max}\{n \ | \ A_{n} \neq 0\}, & \textrm{if} &  A \neq 0. \\ - \infty, & \textrm{if} &  A = 0. \end{array} \right.$$

Now fix a finite graded $S$-module $N\neq 0$. Given $j\geq 0$, let
$$H_{S_{+}}^{j}(N) =  \displaystyle \varinjlim_{k}{\rm Ext}_S^{j}(S/S_+^k,N)$$ be the $j$th local cohomology module of $N$. Recall $H_{S_{+}}^{j}(N)$ is a graded module such that $H_{S_{+}}^{j}(N)_n=0$ for all $n\gg 0$; see \cite[Proposition 15.1.5(ii)]{B-S}. Thus,  ${\rm end}\,H_{S_{+}}^{j}(N)<\infty$.

\begin{definition}\rm The \textit{Castelnuovo-Mumford regularity} of the graded $S$-module $N$ is given by
$$\mathrm{reg}\,N = \mathrm{max}\{{\rm end}\,H_{S_{+}}^{j}(N) + j \, \mid  \, j \geq 0\}.$$
\end{definition}

The following lemma will be very useful to the proof of Theorem \ref{regMulMin}, since it interprets the regularity of Rees modules as a relative reduction number in a suitable setting. It was originally stated in more generality (involving, e.g., $d$-sequences) but here the special case of regular sequences suffices for our purposes.

\begin{lemma}{\rm (cf.\,\cite[Theorem 5.3]{Giral-Planas-Vilanova})} \label{GPV}
Let $R$ be a  ring, $I$ an ideal of $R$ and $M$ a
finite $R$-module. Let $z_{1}, \ldots, z_{s}$ be an $M$-sequence
such that the ideal $J = (z_{1}, \ldots, z_{s})$ is an $M$-reduction of $I$. Let ${\rm r}_{J}(I, M) = r$. Suppose either $s=1$, or else $s\geq 2$ and
    $$(z_{1}, \ldots, z_{i})M \cap I^{r + 1}M = (z_{1}, \ldots, z_{i})I^{r}M \quad \mbox{for \,all} \quad i = 1, \ldots, s - 1.$$ Then, $\mathrm{reg}\,{\mathcal R}(I, M) =  {\rm r}_{J}(I,M)$.
\end{lemma}

We are now ready for the main technical result of this section, which in particular will lead us to a byproduct on Ulrich modules. Note this theorem also gives a generalization of \cite[Proposition
1.2]{Ooishi}, where the situation $\Ic=\Mc$ was treated; more precisely, the condition "$g_{\Delta}(M)=0$"
in \cite{Ooishi} is equivalent to Puthenpurakal's notion of
minimal multiplicity when $\Ic = \Mc$.

\begin{theorem}\label{regMulMin}
Let $(R, \Mc)$ be a local ring with infinite residue field, $M$ a Cohen-Macaulay $R$-module of dimension $t > 0$ and $I$ an $\Mc$-primary ideal of $R$. Let $J = (z_{1}, \ldots,z_{t})$ be a minimal $M$-reduction of $I$. The following assertions are equivalent:
\begin{itemize}
\item[(i)] 	$M$ has minimal multiplicity with respect to $I$;
\item[(ii)] $\mathrm{reg}\,{\mathcal R}(I, M) = \mathrm{reg}\,{\mathcal G}(I,M) = {\rm r}_{J}(I, M)  \leq 1$;
\item[(iii)] ${\rm r}_{J}(I, M)  \leq 1$.
\end{itemize}
\end{theorem}
\begin{proof} First, notice that $z_1, \ldots, z_t$ is a (maximal) $M$-superficial sequence of $I$ by Lemma \ref{connection}. As a consequence, being $M$ Cohen-Macaulay and $I$ $\Mc$-primary,  $z_1, \ldots, z_t$ must be in fact an $M$-sequence according to \cite[Lemma 1.2]{Rossi-Valla}. 
Now, the core of the proof is the implication (i) $\Rightarrow$ (ii), so assume first that (i) holds. In general, we have $\mathrm{reg}\,{\mathcal R}(I, M) = \mathrm{reg}\,{\mathcal G}(I,M)$ (see \cite[Corollary 3]{Zamani14}) and so it remains to prove that $\mathrm{reg}\,{\mathcal R}(I, M) =  {\rm r}_{J}(I,M)$, which we shall accomplish by means of Lemma \ref{GPV}.

Moreover, since $z_1, \ldots, z_t$ is maximal $M$-superficial, Lemma \ref{Puthenchar} yields
	$JIM = I^2M$, i.e., ${\rm r}_{J}(I, M) \leq 1$. Now, to simplify notation, set ${\bf z}_i=z_1, \ldots, z_i$ for $i = 1, \ldots, t - 1$ (note we can assume $t>1$ by Lemma \ref{GPV}). Since clearly $({\bf z}_i)M \cap IM = ({\bf z}_i)M$ for all $i = 1, \ldots, t - 1$, the case ${\rm r}_{J}(I,M) = 0$ is trivial by virtue of Lemma \ref{GPV}. Now suppose ${\rm r}_{J}(I,M) = 1$. Again in view of Lemma \ref{GPV}, all we need to prove is that $$({\bf z}_i)M \cap I^2M = ({\bf z}_i)IM \quad \mbox{for \,all} \quad i = 1, \ldots, t - 1.$$
	
First, it is clear that $({\bf z}_i)IM \subset ({\bf z}_i)M \cap I^2M$. To show the other inclusion, take an arbitrary $f \in ({\bf z}_i)M \cap I^2M$. Because $JIM = I^2M$, we have $$f = z_1 m_1 + \cdots + z_i m_i = z_1 a_1 m'_1 + \cdots + z_t a_t m'_t$$ with $m_j, m'_k \in M$ and $a_k \in I$. Hence $$\overline{z_t a_t m'_t} = \overline{0} \in M/({\bf z}_{t-1})M,$$ and since the sequence is regular on $M$, we have $\overline{a_t m'_t} = \overline{0} \in M/({\bf z}_{t-1})M$, that is, $a_t m'_t = z_1 w_{t, 1} +  \cdots + z_{t - 1} w_{t, t - 1}$ with $w_{t, j} \in M$. Therefore, $f$ can be expressed as
	\begin{equation}\label{calc1}
	z_1 m_1 + \cdots + z_i m_i = z_1(a_1 m'_1 + z_t w_{t, 1}) + \cdots + z_{t - 1}(a_{t - 1} m'_{t - 1} + z_t w_{t, t - 1}),
	\end{equation} whose right-hand side shows $f\in ({\bf z}_{t-1})IM$, thus settling the case $i=t-1$.
	Next, for $i<t-1$, we reduce  (\ref{calc1}) modulo $({\bf z}_{t-2})M$ and apply an analogous argument to the term $z_{t - 1}(a_{t - 1} m'_{t - 1} + z_t w_{t, t - 1})$ in order to obtain 
	\begin{equation}\label{calc2}
	a_{t - 1} m'_{t - 1} + z_t w_{t, t - 1} = z_1 w_{t - 1, 1} + \ldots + z_{t - 2} w_{t - 1, t - 2}
	\end{equation} with $w_{t-1, j} \in M$. Thus, by (\ref{calc1}) and (\ref{calc2}), 
	$$f=z_1(a_1 m'_1 + z_t w_{t, 1} + z_{t - 1} w_{t - 1, 1}) + \cdots + z_{t - 2}(a_{t - 2} m'_{t - 2} + z_t w_{t, t - 2} + z_{t - 1} w_{t - 1, t - 2}).$$ Continuing with the argument, we get an equality
$$f  =  z_1(a_1 m'_1 + z_t w_{t, 1} + \cdots + z_{i + 1} w_{i + 1, 1}) + \cdots + z_{i}(a_{i} m'_{i} + z_t w_{t,i} + \cdots + z_{i + 1} w_{i + 1, i}).$$ Since $a_1, \ldots, a_i, z_{i + 1}, \ldots, z_t \in I$, it follows that $f \in ({\bf z}_i)IM$, as needed.
	
	The implication (ii) $\Rightarrow$ (iii) is obvious. Finally, suppose (iii) holds.  Then $JIM = I^2M$, and we have seen that
	$z_1, \ldots, z_t$ is a maximal $M$-superficial sequence. By Lemma \ref{Puthenchar}, we conclude that $M$ has minimal multiplicity with respect to $I$.  \qed
\end{proof}

\medskip

As a consequence of Theorem \ref{regMulMin}, we determine the regularity of blowup modules of $\Ic$ relative to an Ulrich module. Also, taking $\Ic = \Mc$ the result retrieves part of \cite[Proposition 1.1]{Ooishi}. 

\begin{corollary}\label{corUlr}
	Let $(R, \Mc)$ be a Cohen-Macaulay local ring with infinite residue field and positive dimension, and let $Q$ be as in Convention \ref{convention}. If $M$ is an Ulrich $R$-module with respect to $\Ic$, then
	$$\mathrm{reg}\,{\mathcal R}(\Ic, M) = \mathrm{reg}\,{\mathcal G}(\Ic,M) = {\rm r}_{Q}(\Ic, M)  = 0.$$ The converse holds in case  $M$ is maximal Cohen-Macaulay and $M/\Ic M$ is $R/ \Ic$-free.
\end{corollary}
\begin{proof} First, notice that  $Q$ is an $M$-reduction of $\Ic$, so the number ${\rm r}_{Q}(\Ic, M)$ makes sense. Now, because $M$ is Ulrich with respect to $\Ic$, we have $QM = \Ic M$, which means ${\rm r}_{Q}(\Ic, M)  = 0$. On the other hand, Proposition \ref{UlrichMultMin} and its proof ensure that $M$ has minimal multiplicity with respect to $\Ic$ and that $Q$ is in fact a minimal $M$-reduction of $\Ic$, and so we can apply Theorem \ref{regMulMin} to obtain $\mathrm{reg}\,{\mathcal R}(\Ic, M) = \mathrm{reg}\,{\mathcal G}(\Ic,M) = {\rm r}_Q(\Ic, M)$. The converse is clear. \qed
\end{proof}

\begin{corollary}\label{corUlr2}
	Let $(R, \Mc)$ be a Cohen-Macaulay local ring with infinite residue field and positive dimension. Suppose $\Ic$ is an Ulrich ideal of $R$ but not a parameter ideal. Then,
	$\mathrm{reg}\,{\mathcal R}(\Ic, \Omega^k \Ic) = 0$ for all $k\geq d-1$. If in addition $R$ is Gorenstein, then $$\mathrm{reg}\,{\mathcal R}(\Ic, \lambda (\Omega^k \Ic)) = 0 \quad \mbox{for \,all} \quad k\geq d.$$
\end{corollary}
\begin{proof} As we know, the $R$-module $\Omega^{k+1}(R/\Ic)=\Omega^{k}\Ic$ is Ulrich with respect to $\Ic$ for all $k\geq d-1$. Thus the first part follows from Corollary \ref{corUlr}. If $R$ is Gorenstein then by Corollary \ref{lambda-syz} the $R$-module $\lambda (\Omega^k \Ic)$ is Ulrich with respect to $\Ic$ for all $k\geq d$. Now we again apply Corollary \ref{corUlr}. \qed
\end{proof}

\begin{corollary}\label{reg-irrel}
	Let $(R, \Mc)$ be a 1-dimensional Cohen-Macaulay local ring with infinite residue field. If $\Ic$ is an Ulrich ideal, then $$\mathrm{reg}\,{\mathcal R}(\Ic)_+  = 0.$$
\end{corollary}
\begin{proof} Using Remark \ref{obs.I.urich.module}(ii) and Corollary \ref{corUlr}, we obtain $\mathrm{reg}\,{\mathcal R}(\Ic, \Ic)  = 0$.
On the other hand, we clearly have ${\mathcal R}(\Ic, \Ic)=\bigoplus_{i\geq 0}{\Ic}^{i+1}={\mathcal R}(\Ic)_+$. \qed
\end{proof}		

\begin{example}\rm Consider the local ring $R=K[[x, y]]/(x^2+y^4)$, where $K$ is an infinite field. The ideal $\Ic = (x, y^2)R$ is Ulrich (this is the case $d=1$ and $s=2$ of Example \ref{exemplo}(ii)). Then, Corollary \ref{reg-irrel} gives $\mathrm{reg}\,{\mathcal R}(\Ic)_+  = 0$. To write this graded ideal explicitly, we can use (degree 1) variables $T, U$ over $R$ in order to determine a presentation of the Rees algebra
$${\mathcal R}(\Ic) = R[T, U]/\Kc, \quad \Kc =(xT+y^2U, y^2T-xU, T^2 + U^2), \quad {\mathcal R}(\Ic)_0=R,$$ so that ${\mathcal R}(\Ic)_+=(T, U)R[T, U]/\Kc$.

Now let us use the same example to illustrate the determination of the Hilbert-Samuel polynomial $P_{\Ic}^{\Ic}(k)$. Notice that $\ell_R(R/\Ic)={\rm dim}_K(K[[y]]/(y^2))=2$ and $\nu(\Ic)=2$. By Remark \ref{Hilb-S}(i), we have $P_{\Ic}^{\Ic}(k)=\nu(\Ic)\ell_R(R/\Ic)k=4k$, i.e.,
$$\ell_R(\Ic/{\Ic}^{k+1})=4k \quad \mbox{for \,all} \quad k\gg 0.$$
\end{example}

\section{A detailed example}\label{detailed}

In this last section, we fix formal indeterminates $x, y, z$ over an infinite field $K$ as well as the $2$-dimensional local hypersurface ring $R=K[[x, y, z]]/(x^2+y^2+z^4)$. The ideal $$\Ic = (x, y, z^2)R$$ is Ulrich (this is the case $d=s=2$ of Example \ref{exemplo}(ii)) and not a parameter ideal. Our goal here is to find (explicit) Ulrich $R$-modules with respect to $\Ic$ and study their multiplicities, Chern numbers, and the regularity of the associated blowup modules.

First, $\Ic$ has an infinite (in fact, periodic) minimal $R$-free resolution 
\begin{equation}\label{resol}\cdots \longrightarrow R^4\stackrel{\Phi}{\longrightarrow} R^4\stackrel{\Phi}{\longrightarrow} R^4\stackrel{\Phi}{\longrightarrow} R^4\stackrel{\Psi}{\longrightarrow} R^3\longrightarrow \Ic \longrightarrow 0,\end{equation} where $$\Phi = \left(\begin{array}{cccc}
-z^2  & 0 & -y & x\\
0  & -z^2 & x & y\\
-y  & x & z^2 & 0\\
x & y & 0 & z^2
\end{array}\right), \quad \Psi =\left(\begin{array}{cccc}
-z^2  & 0 & -y & x\\
0  & -z^2 & x & y\\
x  & y & 0 & z^2
\end{array}\right).$$ 

\smallskip

In what follows, as a matter of standard notation, whenever $\varphi$ is a $p\times q$ matrix with entries in $R$, we let ${\rm Im}\,\varphi$ denote the $R$-submodule of $R^p$ generated by the column-vectors of $\varphi$. 
Below we observe a couple of facts.

\smallskip

\begin{itemize}
\item We claim that the $R$-submodules ${\rm Im}\,\Phi \subset R^4$ and ${\rm Im}\,\Psi \subset R^3$ are Ulrich with respect to $\Ic$.
To see this, using Remark \ref{obs.I.urich.module}(iii) we get that $\Omega^{k}\Ic$ is Ulrich with respect to $\Ic$ whenever $k\geq 1$. But in the present case, by (\ref{resol}), these modules are
$$\Omega \Ic = {\rm Im}\,\Psi, \quad \Omega^{k} \Ic = {\rm Im}\,\Phi \quad \mbox{for \,all} \quad k\geq 2,$$ thus showing the claim. Also notice (by the symmetry of $\Phi$) that
$\lambda (\Omega^k \Ic) = \lambda ({\rm Im}\,\Phi)= {\rm Im}\,\Phi^* = {\rm Im}\,\Phi$ for all $k\geq 2$. In particular, ${\rm Im}\,\Phi$ is horizontally linked.

\medskip

\item Let us compute multiplicities and Chern numbers. First, since ${\rm Im}\,\Psi$ is Ulrich with respect to $\Ic$, we must have
${\rm Im}\,\Psi/\Ic {\rm Im}\,\Psi \cong (R/\Ic)^{\nu({\rm Im}\,\Psi)}$. Note $\ell_R(R/\Ic)={\rm dim}_K(K[[z]]/(z^2))=2$. Thus, by Remark \ref{obs.I.urich.module}(i), $${\rm e}^0_{\Ic}({\rm Im}\,\Psi)=\ell_R({\rm Im}\,\Psi/\Ic {\rm Im}\,\Psi)=\nu({\rm Im}\,\Psi)\ell_R(R/\Ic)=4\cdot 2=8.$$ Since $\nu({\rm Im}\,\Phi)=4$ as well, we have ${\rm e}^0_{\Ic}({\rm Im}\,\Phi)=8$. As to the Chern numbers, Corollary \ref{Chern-of-syz} gives ${\rm e}^{1}_{\Ic}(\Omega^{k}\Ic)  = 0$ for all $k\geq 1$. Hence, $${\rm e}^{1}_{\Ic}({\rm Im}\,\Psi)  = {\rm e}^{1}_{\Ic}({\rm Im}\,\Phi) = 0.$$

\medskip

\item For the Castelnuovo-Mumford regularity of blowup modules, Corollary \ref{corUlr2} yields $\mathrm{reg}\,{\mathcal R}(\Ic, \Omega^k \Ic) = 0$ for all $k\geq 1$, and therefore
$$\mathrm{reg}\,{\mathcal R}(\Ic, {\rm Im}\,\Psi) = \mathrm{reg}\,{\mathcal R}(\Ic, {\rm Im}\,\Phi) = 0.$$ Finally, the associated graded ${\mathcal G}(\Ic)$-modules ${\mathcal G}(\Ic, {\rm Im}\,\Psi)$ and ${\mathcal G}(\Ic, {\rm Im}\,\Phi)$ have regularity zero as well (see Corollary \ref{corUlr}), and notice they are Cohen-Macaulay by Corollary \ref{CM-G}.
\end{itemize}

\bigskip

\noindent{\bf Acknowledgments.}  Miranda-Neto (corresponding author) was partially supported by the CNPq-Brazil grants 301029/2019-9 and 406377/2021-9. Queiroz was supported by a CAPES Doctoral Scholarship. The authors are gratefully indebted to the anonymous referee for her/his careful inspection of the paper, including corrections, questions, and interesting suggestions which substantially improved the paper. Last but not least, the authors are also grateful to Shiro Goto (who sadly passed away on July 26, 2022, and to whom this article is dedicated) and to Naoki Endo for some correspondence following the first version of the manuscript.

\end{document}